\documentclass{amsart}
\usepackage[utf8]{inputenc}
\usepackage[british]{babel}
\usepackage{csquotes}
\usepackage{amsmath}
\usepackage{amsfonts}
\usepackage{amsthm}
\usepackage{amssymb}
\usepackage{mathrsfs}
\usepackage{enumerate}
\usepackage{hyperref}
\usepackage[style=alphabetic, maxbibnames=10]{biblatex}
\usepackage[shortlabels]{enumitem}
\usepackage[useregional]{datetime2}
\usepackage[group-separator={,}]{siunitx}

\theoremstyle{plain}
\newtheorem{theorem}{Theorem}[section]
\newtheorem{lemma}[theorem]{Lemma}
\newtheorem{proposition}[theorem]{Proposition}
\newtheorem{corollary}[theorem]{Corollary}
\theoremstyle{definition}
\newtheorem{definition}[theorem]{Definition}
\theoremstyle{remark}
\newtheorem{remark}[theorem]{Remark}

\DeclareMathOperator{\SL}{SL}
\DeclareMathOperator{\GL}{GL}
\DeclareMathOperator{\Sp}{Sp}
\DeclareMathOperator{\Stab}{Stab}
\DeclareMathOperator{\rk}{rk}
\newcommand{\C}{\mathbb C}

\title[On Parabolic Subgroups of Symplectic Reflection Groups]{On Parabolic Subgroups of\\ Symplectic Reflection Groups}
\author{Gwyn Bellamy}
\address{School of Mathematics and Statistics, University of Glasgow, University Place, Glasgow, G12 8QQ, UK}
\email{Gwyn.Bellamy@glasgow.ac.uk}

\author{Johannes Schmitt}
\address{Fachbereich Mathematik, Technische Universität Kaiserslautern, 67663 Kai\-sers\-lau\-tern, Germany}
\email{schmitt@mathematik.uni-kl.de}

\author{Ulrich Thiel}
\address{Fachbereich Mathematik, Technische Universität Kaiserslautern, 67663 Kai\-sers\-lau\-tern, Germany}
\email{thiel@mathematik.uni-kl.de}

\date{\DTMdisplaydate{2022}{11}{23}{-1}}

\begin{document}

\begin{abstract}
  Using Cohen's classification of symplectic reflection groups, we prove that the parabolic subgroups, that is, stabilizer subgroups, of a finite symplectic reflection group are themselves symplectic reflection groups. This is the symplectic analogue of Steinberg's Theorem for complex reflection groups.

  Using computational results required in the proof, we show the non-ex\-ist\-ence of symplectic resolutions for symplectic quotient singularities corresponding to three exceptional symplectic reflection groups, thus reducing further the number of cases for which the existence question remains open.

  Another immediate consequence of our result is that the singular locus of the symplectic quotient singularity associated to a symplectic reflection group is pure of codimension two.
\end{abstract}

\maketitle

\tableofcontents

\section{Introduction}

The study of finite symplectic reflection groups over the complex numbers begun with Cohen's classification \cite{Cohen80}.
More recently, these groups have been studied extensively in the context of symplectic quotient singularities because Verbitsky's Theorem \cite[Theorem 1.1]{Verbitsky00} says that if the quotient \(V/G\), for \(G \subset \Sp(V)\), admits a symplectic resolution then necessarily \(G\) is a symplectic reflection group.
See, for example, the introduction of \cite{BellamySchmittThiel21} for a more detailed overview.

The name ``symplectic reflection group'' invites one to compare them to complex reflection groups. A fundamental result for the latter is Steinberg's Theorem~\cite[Theorem 1.5]{Steinberg64} which says that the parabolic subgroups of complex reflection groups are generated by the complex reflections they contain. Thus, it is natural to ask if the parabolic subgroups of symplectic reflection groups are themselves generated by symplectic reflections.

This question was posed in \cite[Remark (iv)]{Cohen80} and again in \cite[Question 9.1]{BellamySchedler16}. In this paper, we answer the question in the affirmative.

\begin{theorem}
  \label{thm:dim1}
  Let \( (G,V) \) be a symplectic reflection group and choose \(v\in V\).
  Then the stabilizer \(G_v\) of \(v\) in \(G\) is also a symplectic reflection group.
\end{theorem}

Specifically, the theorem says that the stabilizer of \(v\) is generated by those symplectic reflections in \(G\) that fix \(v\).
  It would be interesting to see what other properties of complex reflection groups can be generalised to symplectic reflection groups.

The stabilizer $G_v$ of a vector $v$ is usually called a \textit{parabolic subgroup} of $G$. Therefore, Theorem~\ref{thm:dim1} says that ``every parabolic subgroup of a symplectic reflection group is a symplectic reflection group''.

The proof of Theorem~\ref{thm:dim1} given in Section~\ref{sec:proof} is a case-by-case analysis using the classification of irreducible symplectic reflection groups; see Proposition~\ref{prop:comprefl}, Proposition~\ref{prop:imprim}, and Lemma~\ref{lem:prim}.
The proof of Lemma~\ref{lem:prim} relies on the computation of all maximal parabolic subgroups of seven remaining groups using a computer algebra program; see Section~\ref{sec:data}.

\begin{remark}
  \label{rem:genU}
  By an easy induction (see Corollary~\ref{cor:maxsuff}), one can show that any subgroup of \(G\) that fixes a subset \(U\subset V\) pointwise is also a symplectic reflection group.
\end{remark}

\begin{remark}
  As noted in \cite[Remark 9.2]{BellamySchedler16}, it would be interesting to have a conceptual proof of Theorem~\ref{thm:dim1} that does not rely on the classification of symplectic reflection groups. Such a proof would provide a deeper insight in the nature of symplectic reflection groups.
  The proofs of Steinberg's Theorem for complex reflection groups  given in \cite{Steinberg64}, \cite{Lehrer04}, and \cite[Chapter V, Exercise 8]{Bourbaki68} all make use of alternative (but equivalent) characterisations of these groups. We are not aware of any similar characterisation of symplectic reflection groups that would help here.
\end{remark}

An immediate consequence of the main result is the following; see Corollary~\ref{cor:purecodim2}.

\begin{theorem}
  \label{thm:purecodim2}
  If \(G\) is a symplectic reflection group then the singular locus of \(V/G\) is of pure codimension two.
\end{theorem}

In Section~\ref{sec:res}, we return to the motivating geometric question regarding the existence of symplectic resolutions. Exploiting the computational results necessary for the proof of Lemma~\ref{lem:prim}, we are able to show:

\begin{theorem}
  \label{thm:resintro}
  The symplectic quotients \(\C^6/W(R)\), \(\C^8/W(S_1)\), and \(\C^{10}/W(U)\) do not admit (projective) symplectic resolutions.
\end{theorem}

See Theorem~\ref{thm:res} for the proof.
This leaves 45 cases in which the existence of a symplectic (equivalently crepant) resolution is not yet known.
All remaining cases have rank 4.

A sympletic reflection can be viewed as a particular type of bireflection. It is well-known that the latter appear in the context of complete intersections. More specifically, given a vector space $V$ and finite group $G \subset \GL(V)$, one can ask if $V/G$ is a complete intersection; if this is the case we say that $G$ is a CI-group.

Before a complete (but highly non-trivial) classification of CI-groups was given by Gordeev \cite{Gordeev86} and Nakajima \cite{NakajimaWatanabe84,Nakajima84,Nakajima85}, a concise necessary condition on $G$ for it to be a CI-group was given by Kac and Watanabe. They showed in \cite[Theorem C]{KacWatanabe82} that if $G$ is a CI-group then necessarily every parabolic subgroup of $G$ is generated by either pseudo-reflections or bireflections. If we assume $V$ symplectic and $G \subset \Sp(V)$ then this condition simply says that every parabolic subgroup of $G$ must be a symplectic reflection group. In light of our main result, this simplifies to the statement that if $G$ is a CI-group then $G$ must be a symplectic reflection group.

One might then expect that symplectic reflection groups give rise to a large number of CI-groups. Indeed, in dimension two the resulting Kleinian singularities are all hypersurface and hence every finite subgroup of $\SL_2(\C)$ is a CI-group. However, we show in Proposition~\ref{lem:ci} that if $\dim V > 4$ and $G \subset \Sp (V)$ is symplectically irreducible (not a proper product of symplectic reflection groups) then $G$ is not a CI-group. Thus, symplectic reflection groups do not appear to give any new examples of CI-groups in dimensions larger than 4.
It would be interesting to see if there are any symplectic reflection groups in dimension 4 that are CI-groups.

\subsection*{Acknowledgements}

This paper originated in the workshop on computational group theory in Oberwolfach in August 2021, where the third author mentioned in his talk the computational difficulties concerning parabolic subgroups of symplectic reflection groups \cite{Thiel21}.

The third author would like to thank Alexander Hulpke and Eamonn O’Brien for explaining during the workshop how to overcome these difficulties;
he also thanks Gunter Malle for reminding him of the question of whether the analogue of Steinberg’s Theorem holds for symplectic reflection groups.
We moreover thank Gunter Malle for comments on an early version of this paper.
We would like to thank the referee for a detailed reading of the paper and useful comments.

This work is a contribution to the SFB-TRR 195 `Symbolic Tools in Mathematics and their Application' of the German Research Foundation (DFG).

\section{Preliminaries}
\label{sec:pref}

Throughout this paper, \(V\) is a finite dimensional complex vector space equipped with a symplectic form \(\omega:V\times V\to \C\).
In particular, $V$ is even dimensional. Let
\[\Sp(V) := \{g\in \GL(V)\mid \omega(g.v, g.w) = \omega(v, w)\text{ for all } v, w\in V\}\] be the group of symplectic automorphisms of \(V\).

Let \(G\leq \Sp(V)\) be a (finite) \emph{symplectic reflection group}. This means that \(G\) is generated by \emph{symplectic reflections}; these are the elements \(g\in G\) with \(\rk(g - 1) = 2\).

\begin{lemma}
  \label{lem:fixsymp}
  Let \(v\in V\) and $H = \Stab_G(v)$. Let \(V^{H} \subset V\) be the subspace of points fixed by $H$ and \(W \subset V\) the (unique) $H$-invariant complement to $V^H$ in \(V\).

  Then both \(V^{H}\) and \(W\) are symplectic subspaces and \(W\) is the symplectic orthogonal complement $(V^{H})^\perp$ to $V^{H}$.
\end{lemma}
\begin{proof}
  For \(V^{H}\) to be symplectic we need to show that \(\omega\) restricts to a non-degenerate form on \(V^{H}\). Let \(u\in V^{H}\).
  As \(\omega\) is non-degenerate on \(V\), there exists \(w\in V\) with \(\omega(u, w) \neq 0\).  Then \(\frac{1}{|{H}|}\sum_{g\in H}gw\in V^{H}\) and \[\omega\Big(u, \frac{1}{|{H}|}\sum_{g\in H}gw\Big) = \frac{1}{|{H}|}\sum_{g\in G_v}\omega(gu, gw) = \omega(u, w) \neq 0\;,\] as required.

  We now show that the \(H\)-invariant complement \(W\) of \(V^{H}\) is contained in \((V^{H})^\perp\).
  Let \(w\in W\).
  Then \(w' := \frac{1}{|{H}|}\sum_{g\in H}gw\in W\) by \(H\)-invariance of \(W\), but also \(w'\in V^{H}\) as \(w'\) is fixed by \(H\).
  Hence we must have \(w' = 0\) and so \[\omega(u, w) = \frac{1}{|{H}|}\sum_{g\in H}\omega(gu, gw) = \omega(u, w') = 0\] for all \(u\in V^{H}\) as required.

  Thus, \(W\subset (V^{H})^\perp\) and equality follows directly for dimension reasons.
  In particular, \(W\) is a symplectic subspace.
\end{proof}

\begin{lemma}\label{rem:norank4}
  Theorem~\ref{thm:dim1} holds trivially for triples \((V, \omega, G)\) with \(\dim V \le 4\).
\end{lemma}

\begin{proof}
  Let \(v\), \(H = {G_v}\), and \(W\) be as in Lemma~\ref{lem:fixsymp}.
  Since \(W\) is symplectic, \(H\) is a subgroup of \(\Sp(W)\). We must have \(\dim W < \dim V\) and \(\dim W\) is even.

  If $\dim V = 2$ then there are no non-trivial symplectic subspaces of $V$ and Theorem~\ref{thm:dim1} is vacuous. When \(\dim V = 4\), every proper symplectic subspace has dimension 2 and all finite subgroups of \(\Sp_2(\C) = \SL_2(\C)\) are symplectic reflection groups. Thus,  Theorem~\ref{thm:dim1} holds in this case.
\end{proof}

\section{Proof of the theorem}
\label{sec:proof}

We prove Theorem~\ref{thm:dim1} case-by-case using the explicit classification of symplectically irreducible symplectic reflection groups in \cite{Cohen80}.
Cohen's paper is actually concerned with irreducible quaternion reflection groups, but one can see that these classes of groups are equivalent \cite[p.\ 295]{Cohen80}.

We recall the basic definitions leading to the different cases which we consider.
\begin{definition}
  Let \(V\) be a symplectic vector space and \(G\leq\Sp(V)\) a finite group.
  \begin{enumerate}[(i)]
    \item The group \(G\) is called complex (resp.\ symplectically) \emph{reducible} if there exists a non-trivial decomposition into complex (resp.\ symplectic) \(G\)-invariant subspaces \(V = V_1\oplus V_2\).
      Otherwise, we call \(G\) complex (resp.\ symplectically) \emph{irreducible}.
    \item The group \(G\) is called complex (resp.\ symplectically) \emph{imprimitive} if there exists a non-trivial decomposition \(V = V_1\oplus \cdots\oplus V_n\) into complex (resp.\ symplectic) subspaces \(V_i\subset V\) such that for any \(i\in \{1,\dots, n\}\) and any \(g\in G\) there exists \(j\in\{1,\dots, n\}\) with \(gV_i = V_j\).
      In this case, we call the decomposition \(V = V_1\oplus\cdots\oplus V_n\) a \emph{system of imprimitivity}.
      Otherwise we call \(G\) complex (resp.\ symplectically) \emph{primitive}.
  \end{enumerate}
\end{definition}

\begin{remark}

Assume that $(V,G)$ is a symplectic reflection group. If $V = V_1 \oplus V_2$ is a decomposition of $V$ into a direct sum of $G$-modules, both of which are symplectic subspaces of $V$, then $G = G_1 \times G_2$, with each $G_i \subset \Sp(V_i)$ a symplectic reflection group. Then, the stabilizer $G_v = G_{v_1} \times G_{v_2}$ of a vector $v = v_1 + v_2$ in $V$, with $v_i \in V_i$, is a symplectic reflection group if and only if each $G_{v_i}$ is a symplectic reflection group in $G_i$.
\end{remark}

From now on we assume that the action of \(G\) on \(V\) is symplectically irreducible.

\subsection{Complex reducible groups}

Symplectic irreducibility does not imply complex irreducibility since $V$ might decompose into non-symplectic $G$-submodules.
However, this can only happen if \(G\) preserves a Lagrangian subspace \(\mathfrak{h}\subset V\).
In this case, \(G\) acts on \(\mathfrak{h}\) by complex reflections; see \cite[Section 4.1]{BellamySchedler16}.

Assume now that the action of \(G\) on \(V\) is induced by a complex reflection group \(W\leq \GL(\mathfrak h)\), where \(V\cong \mathfrak h\oplus \mathfrak h^\ast\). For clarity, we will write \(H^\vee\), if we mean the induced action of a subgroup \(H\leq W\) on \(\mathfrak h\oplus \mathfrak h^\ast\).

This particular case of Theorem~\ref{thm:dim1} was already proved as part of \cite[Proposition 7.7]{BrownGordon03}. Since our claim is weaker than the statement of loc.\ cit., a shorter argument suffices. We give it here for the sake of completeness.

\begin{proposition}
  \label{prop:comprefl}
  Theorem~\ref{thm:dim1} holds if \( (G,V) \) is a complex reducible symplectic reflection group.
\end{proposition}
\begin{proof}
  Let \(G\) be defined by a complex reflection group \(W\) as above.
  Let \(v\in V\), so there are \(v_1\in \mathfrak h\) and \(v_2^\ast\in \mathfrak h^\ast\) with \(v = v_1 + v_2^\ast\) and we have \[\Stab_G(v) = \Stab_W(v_1)^\vee \cap \Stab_W(v_2^\ast)^\vee\;.\]
  Since \(\mathfrak h^\ast\) is the dual of the representation \(\mathfrak h\), there exists \(v_2\in \mathfrak h\) with \(\Stab_W(v_2) = \Stab_W(v_2^\ast)\).

  By Steinberg's Theorem \cite[Theorem 1.5]{Steinberg64}, the group \(\Stab_W(v_1)\) is generated by complex reflections.
  We have \[\Stab_W(v_1) \cap \Stab_W(v_2) = \Stab_{\Stab_W(v_1)}(v_2)\;,\] and a second application of Steinberg's Theorem implies that this intersection is generated by complex reflections. Hence, \(\Stab_G(v)\) is generated by symplectic reflections as claimed.
\end{proof}

\subsection{Symplectically imprimitive groups}

We assume from now on that \(G\) is complex irreducible.

In this section, we assume, moreover, that \(G\) is symplectically imprimitive. Let \(V = V_1\oplus\cdots\oplus V_n\), with \(n \geq 2\), be a system of imprimitivity. It is explained in the proof of \cite[Theorem 2.9]{Cohen80} that we have \(\dim V_i = 2\) in this case. Note that Cohen works over the quaternions and shows that $\dim V_i = 1$ as a quaternionic vector space; this is the symplectic analogue of \cite[Proposition 2.2]{Cohen76}.

By Lemma~\ref{rem:norank4}, we may assume \(\dim V > 4\). Then \(G\) is constructed as follows.
Let \(K\leq \SL_2(\C)\) be a finite group and let \(H\leq K\) be a subgroup containing \([K,K]\).
Let \(G_n(K, H)\) be the subgroup of \(K\wr S_n\) consisting of all pairs \((\sigma, k)\), where \(k = (k_1,\dots, k_n)\in K^n\) and \(\sigma\in S_n\), satisfying \(k_1\cdots k_n\in H\); see \cite[Notation 2.8]{Cohen80}.
Then \cite[Theorem 2.9]{Cohen80} says that \(G\) is conjugate to some \(G_n(K, H)\),  where \(\dim V = 2n\).

Notice that the transpositions in \(S_n\) act as symplectic reflections on \(V\); they simply swap two summands in the system of imprimitivity.

\begin{proposition}
  \label{prop:imprim}
  Theorem~\ref{thm:dim1} holds if \(G\) is complex irreducible and symplectically imprimitive.
\end{proposition}
\begin{proof}
We may assume \( G = G_n(K, H)\) as described above.

  Let \(v = (v_1,\dots, v_n)\in V = V_1\oplus\dots\oplus V_n\).
  Now let \(\sigma\in S_n\) such that for the permutation \(v'\) of \(v\) given by \(v'_j := v_{\sigma(i)}\) we have a ``block structure'' \[(v'_1,\dots, v'_{n_0},v'_{n_0 + 1},\dots, v'_{n_0 + n_1},\dots, v'_{n_0+\cdots+n_{r - 1} + 1},\dots, v'_{n_0+\cdots+n_r})\;,\] given by the condition that \(Kv_i' = Kv_j'\) if and only if there exists \(0\leq s\leq r\) with \(\big(\sum_{t = -1}^{s - 1} n_t\big) + 1 \leq i, j\leq \sum_{t = 0}^{s}n_t\), where we set \(n_{-1} := 0\).
  That is, we permute the entries of \(v\) so that elements in the same \(K\)-orbit lie next to each other and the number of elements lying in the same orbit is given by the \(n_i\).
  Without loss of generality, we may assume \(v'_1 = \cdots = v'_{n_0} = 0\).
  After fixing representatives \(w_0,\dots, w_r\) for the occurring orbits, we can find an element \(k\in K^n\) such that \[k.v' = w = (w_0,\dots, w_0,w_1,\dots, w_1,\dots, w_r, \dots, w_r)\;,\] where \((Kw_i)\cap (Kw_j) = \emptyset\) for \(i\neq j\) and \(w_0 = 0\).
  Combining \(\sigma\) and \(k\) hence gives an element \(g\in K\wr S_n\) with \(g.v = w\).

  If an element \(\tau h\in G_n(K, H)\) stabilizes the vector \(w\), then \(\tau\in S_{n_0}\times S_{n_1}\times\cdots \times S_{n_r}\).
  Furthermore, we must have \(h = (h_1,\dots,h_{n_0}, 1,\dots, 1)\) where \(h_1,\dots,h_{n_0}\in K\) with \(h_1\cdots h_{n_0}\in H\).

  Hence \(\Stab_{G_n(K, H)}(v)\) is \((K\wr S_n)\)-conjugate to \(G_{n_0}(K, H)\times S_{n_1}\times\cdots\times S_{n_r}\), which is a (in general, reducible) symplectic reflection group.
  Notice that we may have \(n_i = 1\) for some of the blocks, resulting in trivial factors in the above product.
  The claim now follows as symplectic reflections are preserved under conjugation.
\end{proof}

\subsection{Symplectically primitive groups}

The only remaining case is where \(G\) is complex irreducible and symplectically primitive. Then \(G\) is conjugate to one of the groups classified in \cite[Theorem 3.6]{Cohen80} and \cite[Theorem 4.2]{Cohen80}. Once again, we may assume \(\dim V > 4\) by Lemma~\ref{rem:norank4}. This leaves only seven groups to consider. These are given explicitly via the root systems \(Q\) to \(U\) in \cite[Table II]{Cohen80} and one can check with the help of a computer that all stabilizer subgroups are indeed generated by symplectic reflections.
A list of the groups occurring in this way can be found in Section~\ref{sec:data}.

\begin{lemma}
  \label{lem:prim}
  Theorem~\ref{thm:dim1} holds if \(G\) is complex irreducible and symplectically primitive.
\end{lemma}

This finishes the proof of Theorem~\ref{thm:dim1}.
Finally, we note how Theorem~\ref{thm:dim1} implies the statement in Remark~\ref{rem:genU}.
The proof is the same easy induction as in \cite[Section 7]{Steinberg64} (see also \cite[Corollary 9.51]{LehrerTaylor09}); we repeat it for the reader's convenience.

\begin{corollary}
  \label{cor:maxsuff}
  Let \(V\) be a finite-dimensional (complex) symplectic vector space and \(G \subset \Sp(V)\) a finite symplectic reflection group.
  Let \(U\) be a subset of \(V\).
  Then the subgroup of \(G\) that fixes \(U\) pointwise is also a symplectic reflection group.
\end{corollary}
\begin{proof}
  As the action of \(G\) is linear, we may replace \(U\) be the linear span \(\langle U\rangle\) and assume in the following that \(U\) is a subspace of \(V\).

  Let \(u_1,\dots, u_k\) be a basis of \(U\).
  Recalling that \(\Stab_G(U)\) fixes \(U\) pointwise in this discussion, we have \[\Stab_G(U) = \Stab_G(u_1)\cap\Stab_G(\langle u_2,\dots, u_k\rangle)\;.\]
  Now \[\Stab_G(u_1)\cap \Stab_G(\langle u_2,\dots, u_k\rangle) = \Stab_{\Stab_G(u_1)}(\langle u_2,\dots, u_k\rangle)\] and \(\Stab_G(u_1)\) is a symplectic reflection group by Theorem~\ref{thm:dim1}.
  Hence the claim follows by induction on \(k\).
\end{proof}

\section{Applications}

\subsection{Minimal and maximal parabolic subgroups}

We note the following results on the rank of minimal and maximal parabolic subgroups where minimality and maximality are to be understood with respect to inclusion.

\begin{corollary}
  Let \(G\leq \Sp(V)\) be a finite symplectic reflection group and let \(\{1\}\neq H \leq G\) be a minimal parabolic subgroup.
  Then we have \(\dim V^H = \dim V - 2\), that is, \(H\) is of rank 2.
\end{corollary}
\begin{proof}
  By Corollary~\ref{cor:maxsuff}, the parabolic subgroup \(H\) must contain a symplectic reflection \(s\).
  Set \(K := \Stab_G(V^s)\).
  Then \[\Stab_G(V^H + V^K) = \Stab_H(V^K) = H\cap K\;,\] so \(H = K\) by minimality of \(H\).
  Hence \(\dim V^H = \dim V^s = \dim V - 2\).
\end{proof}

The analogous result for maximal parabolic subgroups is easier and does not require Theorem~\ref{thm:dim1}.
\begin{lemma}
  Let \(G\leq \Sp(V)\) be a finite symplectic reflection group with \(V^G = \{0\}\) and let \(H\leq G\) be a maximal parabolic subgroup.
  Then \(\dim V^H = 2\), that is, \(H\) is of rank \(\dim V - 2\).
\end{lemma}
\begin{proof}
  Let \(S\subset G\) be the set of symplectic reflections.
  Since \(G = \langle S \rangle\), there must exist some \(s \in S\) that is not in \(H\).
  Then \(\dim V^s = \dim V - 2\) and we know
  \[\dim V^s + \dim V^H - \dim V^s \cap V^H = \dim V\] as \(V^s + V^H = V\).
  So if \(\dim V^H > 2\) then \(V^s \cap V^H \neq \{0\}\).
  If this is the case, then let \(K := \Stab_G(V^s\cap V^H)\).
  Since \(V^K \neq \{0\}\), we have \(K \neq G\).
  But \(\langle s,H\rangle \leq K\) so \(H\) is a proper subgroup of \(K\).
  This is a contradiction.
  Thus, \(\dim V^H = 2\).
\end{proof}

\subsection{Codimension of symplectic quotient singularities}
As another application, we have the following result on the singular locus of the symplectic quotient \(V/G\).

\begin{corollary}[{= Theorem~\ref{thm:purecodim2}}]\label{cor:purecodim2}
  If \(G\) is a symplectic reflection group then the singular locus of \(V/G\) is of pure codimension two.
\end{corollary}

\begin{proof}
 Let $Z \subset V/G$ be an irreducible component of the singular locus.
  Choose $p \in Z$ generic and $q \in V$ a point mapping to $p$ under the quotient map $V \to V/G$.
  Let $H$ be the stabilizer of $q$ in $G$.
  Then Luna's slice theorem \cite{Luna73} says that the map $V/H \to V/G$ induced by the map $V \to V$, $v \mapsto v + q$, is \'etale at $0 \in V/H$.
  In particular, the codimension at $p$ of the singular locus in $V/G$ equals the codimension at $0$ of the singular locus in $V/H$.

 Decompose $V = W \oplus V^H$ as a $H$-module. The fact that $p$ is generic in $Z$ means that $0$ is an isolated singularity in $W/H$. But Theorem~\ref{thm:dim1} says that $H$ acts on $W$ as a symplectic reflection group. This implies that the singular locus of $W/H$ has at least one irreducible component of codimension two (for instance, along the points stabilized by a symplectic reflection). We deduce that \(\dim W = 2\) and the irreducible component $Z$ of the singular locus has codimension 2 in $V/G$.
\end{proof}

\section{Complete intersections}

In this section we consider whether $V/G$ is a complete intersection.

\begin{proposition}
    \label{lem:ci}
  Let \(G\) be a symplectically irreducible finite subgroup of $\Sp(V)$ with $\dim V > 4$.  Then \(G\) is not a CI-group.
\end{proposition}

\begin{proof}
    We begin by noting that \cite[Theorem~A]{KacWatanabe82} says that if $G$ is a CI-group then $G$ must be a symplectic reflection group. Assuming this, we can make use of the classifications in \cite{Cohen80} and \cite{Gordeev86}.

  First, let \(G\) be complex reducible, so the action of \(G\) on the symplectic space \(V\) is induced from a complex reflection group \(W\) acting on \(\mathfrak h\) with \(V = \mathfrak h \oplus \mathfrak h^\ast\).
  If \(G\) is a CI-group, we must have \([G,G] = \{1\}\) by \cite[Theorem 3]{Gordeev86}. In other words, \(G\) is abelian. But this can only happen if \(W = G(m, p, 1)\) in the classification \cite{ShephardTodd54}.
  In particular, the group \(W\) is rank 1 and hence \(G\) is rank 2.

Next, assume that \(G\) is complex irreducible and symplectically imprimitive, with system of imprimitivity \(V = V_1\oplus\cdots \oplus V_k\). Recall that this implies \(\dim V_i = 2\) for all $i$, as shown in the proof of \cite[Theorem 2.9]{Cohen80}. But then \(G\) cannot be a CI-group by \cite[{}5.2]{Gordeev86} which says that if $G$ is a CI-group then $\dim V_i = 1$ for all $i$.

  If \(G\) is symplectically primitive and has rank at least six then it must be complex primitive by \cite{Cohen80}. But then it cannot be a CI-group by \cite[Theorem 5]{Gordeev86}. We note that this also follows from our computational results in Section~\ref{sec:data}, together with the arguments in the first paragraph, since all of these groups contain a stabilizer of type \(G(m, p, n)\) with \(n > 1\) which is not a CI-group; this contradicts \cite[Theorem C]{KacWatanabe82}.
\end{proof}

Since the finite subgroups of \(\Sp_2(\C) = \SL_2(\C)\) are well-known to be CI-groups, it remains to understand which symplectic reflection groups of rank 4 are CI-groups.
In theory, one could use the classification of Gordeev and Nakajima \cite{Gordeev86,NakajimaWatanabe84,Nakajima84,Nakajima85} for this, but this appears to be very difficult to do in practice.

Modulo the groups of rank four, Proposition~\ref{lem:ci} answers the first half of \cite[Problem 1]{Fu06} in the case of symplectic quotient singularities.

\section{Symplectic resolutions}
\label{sec:res}

In this section, we return to the question of whether the symplectic quotient \(V/G\), for a symplectically irreducible symplectic reflection group \(G\), admits a symplectic (equivalently crepant) projective resolution. In recent years, much progress has been made in classifying all groups for which this question has a positive answer. The combined work of \cite{EtingofGinzburg02}, \cite{GinzburgKaledin04}, \cite{Gordon03}, \cite{Bellamy09}, \cite{BellamySchedler13}, \cite{BellamySchedler16}, \cite{Yamagishi18}, and \cite{BellamySchmittThiel21} leaves only 48 open cases.

We explain how the computational results in Section~\ref{sec:data} imply that the symplectic quotient associated to the groups \(W(R)\), \(W(S_1)\), and \(W(U)\) (as given in \cite[Table III]{Cohen80}) do not admit projective symplectic resolutions.

\begin{theorem}[{= Theorem~\ref{thm:resintro}}]
  \label{thm:res}
  The symplectic quotients \(\C^6/W(R)\), \(\C^8/W(S_1)\), and \(\C^{10}/W(U)\) do not admit (projective) symplectic resolutions.
\end{theorem}
\begin{proof}
  If there exists a resolution in any of these cases, then the symplectic quotient associated to every parabolic subgroup also admits a symplectic resolution by \cite[Theorem 1.6]{Kaledin03}.

  From the results described in Section~\ref{sec:data}, we see that \(W(R)\) (respectively \(W(S_1)\)) contains a parabolic subgroup conjugate to the complex reflection group \(G(5, 5, 2)\) (respectively conjugate to \(G(3, 3, 3)\)). In both cases, the quotient by this parabolic does not admit a symplectic resolution by \cite{Bellamy09}. Hence neither do the quotients by \(W(R)\) and \(W(S_1)\).

  Finally, \(W(S_1)\) is the stabilizer of a root of \(W(U)\) by \cite[Table III]{Cohen80} (see also Section~\ref{sec:data}). Therefore, this quotient cannot admit a symplectic resolution either.
\end{proof}

\begin{remark}
  With Theorem~\ref{thm:res} in hand, there are now only 45 groups for which the question of existence of a symplectic resolution is not yet decided.
  These are the groups given in \cite[Table 6]{BellamySchmittThiel21}, together with the groups \(W(O_i)\) and \(W(P_i)\) for \(i = 1,2, 3\) in \cite[Table III]{Cohen80}.
  All of them are symplectically primitive, of rank 4. This last fact means that the strategy used in this section cannot be applied, as noted in the proof of Lemma~\ref{rem:norank4}.
\end{remark}

\section{Explicit results on symplectically primitive groups}
\label{sec:data}

We state the explicit results required for Lemma~\ref{lem:prim} and Theorem~\ref{thm:res}, by listing (up to conjugacy) all the maximal parabolic subgroups one finds for the groups in question in Section~\ref{subsec:maxstab}.
We now give an outline of how these groups were computed.

Given a symplectically primitive symplectic reflection group, one computes the conjugacy classes of all subgroups using the computer algebra systems GAP \cite{Gap} or Magma \cite{Magma} with the command \texttt{ConjugacyClassesSubgroups} or \texttt{Subgroups} respectively.
One then checks which of these subgroups are parabolic by determining their fixed space using basic linear algebra and then the stabilizer of the fixed space using the command \texttt{Stabilizer} in either GAP or Magma; if this stabilizer coincides with the group, we have found a parabolic subgroup.
Let \(H\) be one of the parabolic subgroups.
One now computes all symplectic reflections contained in \(H\) by computing the conjugacy classes of \(H\) (using \texttt{ConjugacyClasses} in either GAP or Magma) and checking whether the given representative is a symplectic reflection.
Finally, one checks whether \(H\) is generated by the conjugacy classes of symplectic reflections determined in this way.
As in Corollary~\ref{cor:maxsuff}, it suffices to consider the maximal parabolic subgroups: if \(v, w\in V\), with \(G_w \leq G_v\), then it suffices to check, by induction on rank, that \(G_v\) is generated by symplectic reflections.

Identifying a parabolic subgroup with a group in Cohen's classification is an easy but tedious task using the classification and linear algebra.
As the matrices generating the parabolic subgroups tend to become quite large, we do not do this in detail here; Section~\ref{subsec:ex} serves as an example for these computations.

Magma and GAP files with the necessary code to generate the symplectically primitive symplectic reflection groups can be found on the second author's github page.\footnote{\url{https://github.com/joschmitt/Parabolics}}

\subsection{An example: the group \(W(S_1)\)}
\label{subsec:ex}

As an illustration, we show that there is a parabolic subgroup $H$ of \(W(S_1)\) which is isomorphic (as a symplectic reflection group) to the complex reducible symplectic group coming from the complex reflection group \(G(3, 3, 3)\).

The necessary computer calculations were carried out and cross-checked using the software package Hecke \cite{Hecke} and the computer algebra systems GAP \cite{Gap} and Magma \cite{Magma}.

\subsubsection{The group}

The group $W(S_1)$ is a subgroup of \(\Sp_8(\C)\) of order \(2^8\cdot3^3 = \num{6 912}\).
Like all complex primitive groups, it is given by a root system \cite[Table II]{Cohen80}.
Cohen lists 36 root lines for the group. However, four are enough to generate a group of the correct order. A choice of root lines are
\begin{align*}
  &( 1, i, 0, 0, 0, 0, 1, -i), &&( 1 - i, 1 - i, 0, 0, 0, 0, 0, 0),\\
  &( 1 - i, 0, 1 - i, 0, 0, 0, 0, 0), &&( 1 - i, 0, 0, 1 - i, 0, 0, 0, 0).
\end{align*}
Note that these are the ``complexified'' versions of the vectors over the quaternions given in \cite{Cohen80}.
Thus, the group \(W(S_1)\leq \Sp_8(\C)\) is generated by the symplectic reflection matrices
\begin{align*}
  M_1 := \frac{1}{2}&\left(\begin{smallmatrix}
     1 & i &    &    &   &    & -1 & -i \\
    -i & 1 &    &    &   &    & -i &  1 \\
       &   &  1 &  i & 1 &  i &    &    \\
       &   & -i &  1 & i & -1 &    &    \\
       &   &  1 & -i & 1 & -i &    &    \\
       &   & -i & -1 & i &  1 &    &    \\
    -1 & i &    &    &   &    &  1 & -i \\
     i & 1 &    &    &   &    &  i &  1
  \end{smallmatrix}\right),
  &&M_2 := \left(\begin{smallmatrix}
       & -1 &    &    &    &    &    &    \\
    -1 &    &    &    &    &    &    &    \\
       &    &  1 &    &    &    &    &    \\
       &    &    &  1 &    &    &    &    \\
       &    &    &    &    & -1 &    &    \\
       &    &    &    & -1 &    &    &    \\
       &    &    &    &    &    &  1 &    \\
       &    &    &    &    &    &    &  1
  \end{smallmatrix}\right),
\end{align*}
\begin{align*}
  M_3 := &\left(\begin{smallmatrix}
       &    & -1 &    &    &    &    &    \\
       &  1 &    &    &    &    &    &    \\
    -1 &    &    &    &    &    &    &    \\
       &    &    &  1 &    &    &    &    \\
       &    &    &    &    &    & -1 &    \\
       &    &    &    &    &  1 &    &    \\
       &    &    &    & -1 &    &    &    \\
       &    &    &    &    &    &    &  1
  \end{smallmatrix}\right),
  &&M_4 := \left(\begin{smallmatrix}
       &    &    & -1 &    &    &    &    \\
       &  1 &    &    &    &    &    &    \\
       &    &  1 &    &    &    &    &    \\
    -1 &    &    &    &    &    &    &    \\
       &    &    &    &    &    &    & -1 \\
       &    &    &    &    &  1 &    &    \\
       &    &    &    &    &    &  1 &    \\
       &    &    &    & -1 &    &    &
  \end{smallmatrix}\right),
\end{align*}
through these root lines.

\subsubsection{The parabolic subgroup}

Let \(v := (0, 0, 0, 1, -\alpha, \alpha, -\alpha, \alpha + 1)^\top\in \C^8\), where \(\alpha := \frac{1}{2}(i - 1)\). Let \(H\leq W(S_1)\) be the stabilizer of \(v\). Using the command \texttt{Stabilizer} in either GAP or Magma one can compute this group: \[H = \langle M_2, M_3M_1M_3, M_4M_1M_4\rangle\;.\]
The space \(V^H \subset \C^8\) of vectors fixed by \(H\) is generated by \(v\) and \((1, -1, 1, 0, \alpha + 1, -\alpha - 1, \alpha + 1, 3\alpha)^\top\).
Its \(H\)-invariant complement \(W\) has a basis given by the columns \(w_1,\dots, w_6\in \C^8\) of the matrix
\[
  \left(\begin{smallmatrix}
    \zeta^2 + \zeta + 1 & \zeta^3 + \zeta^2 - \zeta - 2 & \zeta & -\zeta^3 + \zeta^2 + \zeta - 2 & \zeta^2 - \zeta + 1 & -\zeta^3 + \zeta \\
    -\zeta^3 - \zeta^2 + \zeta + 2 & -\zeta^2 - \zeta - 1 & -\zeta & -\zeta^2 + \zeta - 1 & \zeta^3 - \zeta^2 - \zeta + 2 & \zeta^3 - \zeta \\
    -\zeta^3 - 2\zeta^2 + 1 & -\zeta^3 - 2\zeta^2 + 1 & \zeta & \zeta^3 - 2\zeta^2 + 1 & \zeta^3 - 2\zeta^2 + 1 & -\zeta^3 + \zeta \\
    0 & 0 & -2\zeta^3 + \zeta & 0 & 0 & \zeta^3 + \zeta \\
    -\zeta^2 - \zeta + 1 & -\zeta^3 + \zeta^2 + \zeta & \zeta^3 + \zeta^2 & \zeta^3 - \zeta^2 - \zeta & \zeta^2 + \zeta - 1 & -\zeta^3 + \zeta^2 - 1 \\
    \zeta^3 - \zeta^2 - \zeta & \zeta^2 + \zeta - 1 & -\zeta^3 - \zeta^2 & -\zeta^2 - \zeta + 1 & -\zeta^3 + \zeta^2 + \zeta & \zeta^3 - \zeta^2 + 1 \\
    \zeta^3 - 1 & \zeta^3 - 1 & \zeta^3 + \zeta^2 & -\zeta^3 + 1 & -\zeta^3 + 1 & -\zeta^3 + \zeta^2 - 1 \\
    0 & 0 & -\zeta^3 - \zeta^2 + 2\zeta + 2 & 0 & 0 & -\zeta^3 - \zeta^2 + 2\zeta - 1 \\
  \end{smallmatrix}\right)
\]
where \(\zeta\in \C\) is a primitive 12-th root of unity such that \(\zeta^3 = i\).

By changing the basis from \(\C^8\) to \(W \oplus V^H\) and restricting to \(W\) we may identify \(H\) with the subgroup \(H_W\) of \(\Sp(W)\) generated by the matrices
\[\left(\begin{smallmatrix}
    & 1 &   &   &   &   \\
  1 &   &   &   &   &   \\
    &   & 1 &   &   &   \\
    &   &   &   & 1 &   \\
    &   &   & 1 &   &   \\
    &   &   &   &   & 1
  \end{smallmatrix}\right),
  \left(\begin{smallmatrix}
    &   & 1 &   &   &   \\
    & 1 &   &   &   &   \\
  1 &   &   &   &   &   \\
    &   &   &   &   & 1 \\
    &   &   &   & 1 &   \\
    &   &   & 1 &   &
  \end{smallmatrix}\right),
  \left(\begin{smallmatrix}
              &   & \zeta^2 - 1 &             &   &          \\
              & 1 &             &             &   &          \\
    -\zeta^2  &   &             &             &   &          \\
              &   &             &             &   & -\zeta^2 \\
              &   &             &             & 1 &          \\
              &   &             & \zeta^2 - 1 &   &
  \end{smallmatrix}\right).
  \]
The basis of \(W\) was chosen so that the symplectic form on \(W \) is, up to a constant, given by the matrix
\[
  \begin{pmatrix} & I_3 \\ -I_3 & \end{pmatrix}.
\]

One can see directly that \(H_W\) leaves the subspace \(\langle w_1, w_2, w_3\rangle\) invariant and that this subspace is Lagrangian. Hence \(H_W\) is a complex reducible, but symplectically irreducible, group coming from a complex reflection group in \(\GL(\langle w_1, w_2, w_3\rangle)\).
Since the complex reflection group has rank 3 and order 54 it must be conjugate to \(G(3, 3, 3)\) in the classification \cite{ShephardTodd54}.

\subsection{Maximal parabolic subgroups}
\label{subsec:maxstab}

In this section, we list the maximal parabolic subgroups up to conjugation.

\subsubsection{\(W(Q)\)}

The group \(W(Q)\) is a subgroup of \(\Sp_6(\C)\) of order \(2^6\cdot3^3\cdot7 = \num{12096}\).
It is generated by the symplectic reflections corresponding to the root lines
\begin{align*}
  &(2, 0, 0, 0, 0, 0)\;, && \frac{1}{2}(2i, 2i, -i + 1, 0, 0, i\sqrt{5} - 1)\;, \\
  &\frac{1}{2}(2i, 2, i + 1, 0, 0, i + \sqrt{5})\;, && \frac{1}{2}(2, 2i, i + 1, 0, 0, i + \sqrt{5}).
\end{align*}

The maximal parabolic subgroups are each conjugate to \(H_1 := G(3, 3, 2)\) or \(H_2 := G(4, 2, 2)\).
They stabilize the following vectors:
\[\begin{array}{c | l}
   & v \\
  \hline
  H_1 & (1, 0, 0, \alpha, \beta, 0) \\
  H_2 & (1, 0, 1, \alpha, 2\beta, \alpha)
\end{array}\]
where \(\alpha := \frac{1}{6}(i\sqrt{5} + i - \sqrt{5} + 1)\) and \(\beta := \frac{1}{3}(-i + \sqrt{5})\).

\subsubsection{\(W(R)\)}

The group \(W(R)\) is a subgroup of \(\Sp_6(\C)\) of order \(2^8\cdot3^3\cdot5^2\cdot7 = \num{1209600}\).
It is generated by the symplectic reflections corresponding to the root lines
\begin{align*}
  &(2, 0, 0, 0, 0, 0)\;, && \frac{1}{2}(i + 1, i - 1, 0, -i\sqrt{5} + 1, -i - \sqrt{5}, 0)\;, \\
  &\frac{1}{2}(0, 0, i + 1, -i, 1, i + \sqrt{5})\;, && \frac{1}{2}(0, 2i, i + \sqrt{5}, 2, 0, -i - 1)\;.
\end{align*}

The maximal parabolic subgroups are conjugate to \(H_1 := G(3, 3, 2)\), \(H_2 := G(5, 5, 2)\), or \(H_3 := G(\mathsf D_2, \mathsf C_2, 1)\).
They stabilize the following vectors:
\[\begin{array}{c | l}
   & v \\
  \hline
  H_1 & (0, 1, \sqrt{5} - 1, \sqrt{5} - 1, \frac{1}{2}(i\sqrt{5} - i + \sqrt{5} - 3, 0) \\
  H_2 & (0, 1, \frac{1}{2}(i\sqrt{5} + i - 2), \frac{1}{2}(i\sqrt{5} + i + \sqrt{5} + 1), 1,\frac{1}{2}(-2i + \sqrt{5} + 1)) \\
  H_3 & (0, 1, \frac{1}{2}(\sqrt{5} + 3), \frac{1}{2}(-i\sqrt{5} - i + \sqrt{5} + 1), 1, \frac{1}{2}(-i\sqrt{5} - 3i))
\end{array}\]

\subsubsection{\(W(S_1)\)}

The group \(W(S_1)\) is a subgroup of \(\Sp_8(\C)\) of order \(2^8\cdot3^3 = \num{6912}\).
It is generated by the symplectic reflections corresponding to the root lines
\begin{align*}
  & (1, i, 0, 0, 0, 0, 1, -i)\;, && (-i + 1, -i + 1, 0, 0, 0, 0, 0, 0)\;,\\
  & (-i + 1, 0, -i + 1, 0, 0, 0, 0, 0)\;, && (-i + 1, 0, 0, -i + 1, 0, 0, 0, 0)\;.
\end{align*}

The maximal parabolic subgroups are conjugate to \(H_1 := \mathsf C_2\times \mathsf C_2 \times\mathsf C_2\), \(H_2 := G(2, 2, 3)\), or \(H_3 := G(3, 3, 3)\).
They stabilize the following vectors:
\[\begin{array}{c | l}
   & v \\
  \hline
  H_1 & (1, 0, 0, -1, 0, 0, 0, 0) \\
  H_2 & (0, 1, i, 0, 0, 0, 0, 0) \\
  H_2 & (1, i, i, -1, 0, 0, 0, 0) \\
  H_2 & (0, 0, 1, 0, 0, 0, 0, 0) \\
  H_2 & (0, 1, 0, 0, 0, 0, 1, 0) \\
  H_3 & (0, 0, 0, 1, \frac{1}{2}(1 - i), \frac{1}{2}(i - 1), \frac{1}{2}(1 - i), \frac{1}{2}(i + 1))
\end{array}\]
Note that multiple occurrences of $H_2$ in the above table means that there are distinct maximal parabolic subgroups which are conjugate in \(\GL_8(\C)\), but not in \(W(S_1)\).

\subsubsection{\(W(S_2)\)}

The group \(W(S_2)\) is a subgroup of \(\Sp_8(\C)\) of order \(2^{10}\cdot3^4 = \num{82944}\).
It is generated by the symplectic reflections corresponding to the root lines
\begin{align*}
  & (1, i, 0, 0, 0, 0, 1, -i)\;, && (-i + 1, -i + 1, 0, 0, 0, 0, 0, 0)\;, \\
  & (-i + 1, 0, -i + 1, 0, 0, 0, 0, 0)\;, && (2, 0, 0, 0, 0, 0, 0, 0)\;.
\end{align*}
The maximal parabolic subgroups are conjugate to \(H_1 := \mathsf C_2 \times G(3, 3, 2)\), \(H_2 := G(2, 2, 3)\), \(H_3 := G(2, 1, 3)\), \(H_4 := G(3, 3, 3)\), and \(H_5 := G(4, 4, 3)\).
They stabilize the following vectors:
\[\begin{array}{c | l}
   & v \\
  \hline
  H_1 & (1, -1, -1, 0, 0, 0, 0, 0) \\
  H_2 & (1, -1, 0, 0, 0, 0, i - 1, 0) \\
  H_3 & (0, 1, 0, 0, 0, 0, 0, 0) \\
  H_4 & (1, 0, 1, 0, -1, i - 1, -i, 0) \\
  H_5 & (1, -i, 0, 0, 0, 0, 0, 0)
\end{array}\]

\subsubsection{\(W(S_3)\)}

The group \(W(S_3)\) is a subgroup of \(\Sp_8(\C)\) of order \(2^{13}\cdot3^4\cdot5 = \num{3317760}\).
It is generated by the symplectic reflections corresponding to the root lines
\begin{align*}
  & (1, i, 0, 0, 0, 0, 1, -i)\;, && (-i + 1, -i + 1, 0, 0, 0, 0, 0, 0)\;,\\
  & (-i + 1, 0, -i + 1, 0, 0, 0, 0, 0)\;, && (2, 0, 0, 0, 0, 0, 0, 0)\;,\\
  & (-i + 1, 0, 0, 0, 0, -i + 1, 0, 0)\;.
\end{align*}
The maximal parabolic subgroups are conjugate to \(H_1 := \mathsf C_2 \times G(3, 3, 2)\), \(H_2 := G(2, 2, 3)\), \(H_3 := G(3, 3, 3)\), and \(H_4 := G_3(\mathsf D_2, \mathsf C_2)\).
They stabilize the following vectors:
\[\begin{array}{c | l}
   & v \\
  \hline
  H_1 & (1, -i, 0, 0, 0, 0, 3, i) \\
  H_2 & (0, 0, 2, 0, -1, i, i, 1) \\
  H_3 & (0, 1, 0, -1, -1, i - 1, i, 0) \\
  H_4 & (1, 0, 0, 1, 0, 0, 0, 0)
\end{array}\]

\subsubsection{\(W(T)\)}

The group \(W(T)\) is a subgroup of \(\Sp_8(\C)\) of order \(2^8\cdot3^4\cdot5^3 = \num{2592000}\).
It is generated by the symplectic reflections corresponding to the root lines
\begin{align*}
  & (-\zeta^3 + \zeta^2 + 1, \zeta^3 - \zeta^2, -1, 0, 0, 0, 0, 0)\;, && (1, 0, 0, 0, 0, 0, 0, 0)\;,\\
  & (1, 1, 1, 1, 0, 0, 0, 0)\;, && (1, i, 0, 0, 0, 0, -1, i)\;,
\end{align*}
where \(\zeta\) is a primitive 10-th root of unity.

The maximal parabolic subgroups are conjugate to \(H_1 := \mathsf C_2 \times G(3, 3, 2)\), \(H_2 := \mathsf C_2 \times G(5, 5, 2)\), \(H_3 := G(2, 2, 3)\), \(H_4 := G(3, 3, 3)\), \(H_5 := G_{23}\), and \(H_6 := G(5, 5, 3)\).
They stabilize the following vectors:
\[\begin{array}{c | l}
   & v \\
  \hline
  H_1 & (0, 0, 0, 0, 1, \frac{1}{2}(-\zeta^3 + \zeta^2 + 1), -\zeta^3 + \zeta^2 + \frac{1}{2}, \frac{1}{2}(\zeta^3 - \zeta^2)) \\
  H_2 & (0, 0, 0, 0, 1, -\zeta^3 + \zeta^2 + 1, -\zeta^3 + \zeta^2, 2\zeta^3 - 2\zeta^2 - 2) \\
  H_3 & (1, 0, \zeta^3 + \zeta^2 + 1, -3\zeta^3 + 3\zeta^2 + 4, i\zeta^3 - i\zeta^2 - \zeta^3 + \zeta^2 + 1, \\
      & -2i\zeta^3 - 2i\zeta^2 - 2i - 2\zeta^3 + 2\zeta^2 + 4, i + \zeta^3 - \zeta^2, -i\zeta^3 + i\zeta^2 + i + 1) \\
  H_3 & (0, 0, 0, 0, 1, -\zeta^3 + \zeta^2 + 2, 0, \zeta^3 - \zeta^2 - 3) \\
  H_4 & (1, 0, -\zeta^3 + \zeta^2 + 1, \zeta^3 - \zeta^2, -i, i\zeta^3 - i\zeta^2 - i + \zeta^3 - \zeta^2 - 1, \\
      & i\zeta^3 - i\zeta^2 + 1, \zeta^3 - \zeta^2) \\
  H_5 & (0, 1, \frac{1}{5}(4i\zeta^3- 4i\zeta^2 - 2i + 3\zeta^3 - 3\zeta^2 - 4), (2i\zeta^3 - 2i\zeta^2 - 6i - \zeta^3 + \zeta^2 + 3),\\
      & \frac{1}{5}(4i\zeta^3 - 4i\zeta^2 - 2i - 2\zeta^3 + 2\zeta^2 + 6), (3i\zeta^3 - 3i\zeta^2 - 4i + \zeta^3 - \zeta^2 - 3),\\
      & \frac{1}{5}(i\zeta^3 - i\zeta^2 - 3i + 2\zeta^3 - 2\zeta^2 - 1)) \\
  H_6 & (0, 1, \zeta^3 - \zeta^2, \zeta^3 - \zeta^2 + 1, -2i\zeta^3 + 2i\zeta^2, -i\zeta^3 + i\zeta^2 - \zeta^3 + \zeta^2 + 1,\\
      & -i\zeta^3 + i\zeta^2 + i - 1, i - \zeta^3 + \zeta^2)
\end{array}\]

Note that there are two distinct maximal parabolic subgroups which are conjugate in \(\GL_8(\C)\), but not in \(W(T)\).

\subsubsection{\(W(U)\)}

The group \(W(U)\) is a subgroup of \(\Sp_{10}(\C)\) of order \(2^{11}\cdot3^5\cdot5\cdot11 = \num{27371520}\).
It is generated by the symplectic reflections corresponding to the root lines
\begin{align*}
  &(2, 0, 0, 0, 0, 0, 0, 0, 0, 0)\;,\\
  &(0, 2, i - 1, i - 1, 2, 0, 0, -i + 1, -i + 1, 0)\;,\\
  &(0, 2, i - 1, -i - 1, 2i, 0, 0, i - 1, -i - 1, 0)\;,\\
  &(0, 2, -i - 1, i - 1, 0, 0, 0, i + 1, i - 1, 2)\;,\\
  &(2, i - 1, i - 1, 2, 0, 0, -i + 1, -i + 1, 0, 0)\;.
\end{align*}

The maximal parabolic subgroups are conjugate to \(H_1 := \mathsf C_2 \times G(2, 2, 3)\), \(H_2 := \mathsf C_2 \times G(3, 3, 3)\), \(H_3 := \mathfrak S_5\), \(H_4 := G(3, 3, 4)\), and \(H_5 := W(S_1)\).
They stabilize the following vectors:
\[\begin{array}{c | l}
   & v \\
  \hline
  H_1 & (2, 0, i - 1, 0, -i + 3, 2i + 2, -2, -i - 1, 0, i - 1) \\
  H_2 & (0, 2, i - 1, 0, -i - 1, -6i, 0, i + 1, 0, -i + 1) \\
  H_3 & (1, 0, -2i + 1, i + 1, -i + 1, -i, 2i, i, -2i, 0) \\
  H_4 & (2, 0, 0, -i - 1, -i - 1, 0, -i + 1, -2i, -i - 1, i + 1) \\
  H_5 & (2, 0, i + 1, 0, i - 1, 0, -2i, i - 1, 0, i + 1)
\end{array}\]

\clearpage
\printbibliography

\end{document}